\documentclass[a4paper, 11pt]{article}

\RequirePackage[OT1]{fontenc}
\RequirePackage{amsthm,amsmath,float,amssymb, enumerate,setspace}
\RequirePackage[round,sort&compress, authoryear]{natbib}
\RequirePackage[colorlinks,citecolor=blue,urlcolor=blue]{hyperref}

\usepackage{eufrak,mathrsfs}
\RequirePackage{bbm,amsthm}
\RequirePackage{amssymb}

\RequirePackage{enumitem}
\RequirePackage{bigints}
\RequirePackage{mathtools}
\RequirePackage[stable]{footmisc}

\theoremstyle{plain}

\theoremstyle{remark}

\theoremstyle{plain}

 \newcommand{\bea}{\begin{eqnarray}}
\newcommand{\ena}{\end{eqnarray}}
\newcommand{\beas}{\begin{eqnarray*}}
\newcommand{\enas}{\end{eqnarray*}} 

\newtheorem{MLEexist}{Theorem}[section]

\newtheorem{Lemma}{Lemma}[section]
\newtheorem{Remark}{Remark}[section]
\newtheorem{Corollary}{Corollary}[section]

\title{Bounds for the normal approximation of the maximum likelihood estimator from $m$-dependent random variables} 
\author{Andreas Anastasiou
\\
\small{London School of Economics and Political Science, London, UK}
\\
\small{E-mail: A.Anastasiou@lse.ac.uk} 
}
\date{September 2016}

\begin{document}
\pagenumbering{roman}
\pagenumbering{arabic}
\maketitle

\begin{abstract}
\indent The asymptotic normality of the Maximum Likelihood Estimator (MLE) is a long established result. Explicit bounds for the distributional distance between the distribution of the MLE and the normal distribution have recently been obtained for the case of independent random variables. In this paper, a local dependence structure is introduced between the random variables and we give upper bounds which are specified for the Wasserstein metric.\end{abstract}

{\it Key words}: Maximum likelihood estimator; dependent random variables; normal approximation; Stein's method

\section{Introduction}
\label{sec:intro}
The asymptotic normality of Maximum Likelihood Estimators (MLEs) was first discussed in \cite{Fisher}. It is a fundamental qualitative result and a cornerstone in mathematical statistics. The aim of assessing the quality of this normal approximation in the case of independent random variables has recently been accomplished and this topic keeps attracting researchers' interest; see for example \cite{Anastasiou_Reinert}, \cite{Anastasiou_Ley}, \cite{Anastasiou_single} and \cite{Pinelis}. In this paper, the independence assumption is relaxed and we assess the normal approximation of the MLE under the presence of a local dependence structure between the random variables; for limit theorems for sums of $m$-dependent random variables see \cite{Heinrich}. For our purpose, we partly employ a powerful probabilistic technique called Stein's method, first introduced by Charles Stein in \cite{Stein1972}, while the monograph \cite{Steinmonograph} explains in detail the method and it is in our opinion the most notable contribution. Stein's method is used to assess whether a random variable, $W$, has a distribution close to a target distribution. In this paper, the normal approximation related to the MLE is assessed in terms of the Wasserstein distance. If $F, G$ are two random variables with values in $\mathbb{R}$ and
\begin{equation}
\label{classWasserstein}
H_W = \left\lbrace h:\mathbb{R}\rightarrow\mathbb{R}:|h(x) - h(y)|\leq|x-y|\right\rbrace
\end{equation}
is the set of Lipschitz functions with constant equal to one, then the Wasserstein distance between the laws of $F$ and $G$ is
\begin{equation}
\nonumber d_{W}\left(F,G\right) = \sup \left\lbrace|{\rm E}[h(F)] - {\rm E}[h(G)]|: h \in H_{W}\right\rbrace.
\end{equation}
A general approach is first developed to get upper bounds on the Wasserstein distance between the distribution of the suitably scaled MLE and the standard normal distribution; here Stein's method is used for some results. The special case of independent random variables is briefly discussed, while an example of normally distributed locally dependent random variables serves as an illustration of the main results.

The notion of local dependence is introduced before the Stein's method result that is used in the case of locally dependent random variables is given. An $m$-dependent sequence of random variables $\left\lbrace X_i, i \in \mathbb{N}\right\rbrace$ is such that for each $i \in \mathbb{N}$ the sets of random variables $\left\lbrace X_j, j\leq i \right\rbrace$ and $\left\lbrace X_j, j>i+m \right\rbrace$ are independent. The Stein's method result for the case of locally dependent random variables is based on the local dependence condition (LD); for a set of random variables $\left\lbrace \xi_i, i=1,2,\ldots,n \right\rbrace$, for any $A \subset \left\lbrace 1,2,\ldots,n \right\rbrace$ we define
\begin{equation}
\nonumber A^{c} = \left\lbrace i \in \left\lbrace 1,2,\ldots,n \right\rbrace:i \notin A \right\rbrace, \qquad\;\; \xi_{A} = \left\lbrace \xi_i:i\in A \right\rbrace.
\end{equation}
Then,
\begin{itemize}
\item[(LD)] For each $i \in \left\lbrace 1,2,\ldots,n \right\rbrace$ there exist $A_i\subset B_i\subset \left\lbrace 1,2,\ldots,n \right\rbrace$ such that $\xi_i$ is independent of $\xi_{A_i^c}$ and $\xi_{A_i}$ is independent of $\xi_{B_i^c}$.
\end{itemize}
Whenever this condition holds,
\begin{equation}
\label{notation_local}
\eta_i = \sum_{j\in A_i}^{}\xi_j, \qquad\qquad \tau_i = \sum_{j\in B_i}^{}\xi_j.
\end{equation}
Lemma \ref{Theorem_Chen_local} below gives an upper bound on the Wasserstein distance between the distribution of a sum of $m$-dependent random variables satisfying (LD) and the normal distribution. The random variables are assumed to have mean zero with the variance of their sum being equal to one. The proof of the lemma is beyond the scope of the thesis and can be found in \cite[p.134]{Chen_book}.
\vspace{0.1in}
\begin{Lemma}
\label{Theorem_Chen_local}
Let $\left\lbrace \xi_i, i=1,2,\ldots,n \right\rbrace$ be a set of random variables with mean zero and ${\rm Var}(W) = 1$, where $W = \sum_{i=1}^{n}\xi_i$. If (LD) holds, then with $\eta_i$ and $\tau_i$ as in \eqref{notation_local},
\begin{equation}
\label{bound_Chen_local}
d_{W}(W,Z) \leq 2\sum_{i=1}^{n}\left({\rm E}|\xi_i\eta_i\tau_i| + \left|{\rm E}(\xi_i\eta_i)\right|{\rm E}|\tau_i|\right) + \sum_{i=1}^{n}{\rm E}\left|\xi_i\eta_i^2\right|.
\end{equation}
\end{Lemma}
\vspace{0.05in}
Now the notation used throughout the paper is explained. First of all, $\theta$ is a scalar unknown parameter found in a parametric statistical model. Let $\theta_0$ be the true (still unknown) value of the parameter $\theta$ and let $\Theta \subset \mathbb{R}$ denote the parameter space, while $\boldsymbol{X} = (X_1, X_2, \ldots, X_n)$ for $\left\lbrace X_i, i=1,2,\ldots,n \right\rbrace$ an $m$-dependent sequence of identically distributed random variables. The joint density function of $X_1, X_2, \ldots, X_n$ is
\begin{align}
\nonumber f(\boldsymbol{x}|\theta) = L(\theta;\boldsymbol{x}) &= f(x_1;\theta)f(x_2|x_1;\theta)\ldots f(x_n|x_{n-1},\ldots,x_{n-m};\theta)\\
\nonumber & =f(x_1;\theta)\prod_{i=2}^{n}f(x_i|x_{i-1},\ldots,x_{m_i^*};\theta),
\end{align}
where $m_i^*=\max\left\lbrace i-m,1 \right\rbrace$. The likelihood function is $L(\theta; \boldsymbol{x}) = f(\boldsymbol{x}|\theta)$. Its natural logarithm, called the log-likelihood function is denoted by $l(\theta;\boldsymbol{x})$. Derivatives of the log-likelihood function, with respect to $\theta$, are denoted by $l'(\theta;\boldsymbol{x}),l''(\theta;\boldsymbol{x}),\ldots, l^{(j)}(\theta;\boldsymbol{x})$, for $j$ any integer greater than 2. The MLE is denoted by $\hat{\theta}_n(\boldsymbol{X})$. For many models the MLE exists and is unique; this is known as the `regular' case. For a number of statistical models, however, uniqueness or even existence of the MLE is not secured; see \cite{Billingsley} for an example of non-uniqueness. Assumptions that ensure existence and uniqueness of the MLE are given in \cite{Makelainen}.

In Section \ref{sec:bound_reg_not_satisfied} we explain, for locally dependent random variables, the process of finding an upper bound on the Wasserstein distance between the distribution of the suitably scaled MLE and the standard normal distribution. The quantity we are interested in is split into two terms with the one being bounded using Stein's method and the other using alternative techniques based mainly on Taylor expansions. After obtaining the general upper bound, we comment on how our bound behaves for i.i.d. ($m=0$) random variables and this specific result is compared to already existing bounds for i.i.d. random variables as given in \cite{Anastasiou_Reinert}. The main result of this paper is applied in Section \ref{sec:example_local_normal} to the case of 1-dependent normally distributed random variables.
\section{The general bound}
\label{sec:bound_reg_not_satisfied}
The purpose is to obtain an upper bound on the Wasserstein distance between the distribution of an appropriately scaled MLE and the standard normal distribution. The results of Lemma \ref{Theorem_Chen_local} will be applied to a sequence $\left\lbrace \xi_i, i=1,2,\ldots,n \right\rbrace$ of $2m$-dependent random variables. We denote by
\begin{equation}
\begin{aligned}
\nonumber & M_{1j}:= \max\left\lbrace 1,j-2m \right\rbrace \qquad\qquad M_{2j}:= \min\left\lbrace n,j+2m \right\rbrace\\
\nonumber & K_{1j}:= \max\left\lbrace 1,j-4m \right\rbrace \qquad\qquad \;K_{2j}:= \min\left\lbrace n,j+4m \right\rbrace.
\end{aligned}
\end{equation}
In addition, the dependency neighbourhoods, $A_j$ and $B_j$, as defined in (LD) are
\begin{equation}
\label{A_iB_i}
\begin{aligned}
& A_j = \left\lbrace M_{1j}, M_{1j} + 1, \ldots, M_{2j} - 1, M_{2j}\right\rbrace, \qquad B_j = \left\lbrace K_{1j}, K_{1j}+1, \ldots, K_{2j}-1, K_{2j}\right\rbrace.
\end{aligned}
\end{equation}
 Having that $\forall i \in \left\lbrace 1,2,\ldots,n \right\rbrace, \left|A_i\right|$ and $\left|B_i\right|$ denote the number of elements in the sets $A_i$ and $B_i$, respectively, then
\begin{equation}
\nonumber \left|A_i\right| \leq 4m+1, \qquad\quad \left|B_i\right| \leq 8m+1.
\end{equation}
{\raggedright{We work under the following assumptions:}}
\begin{itemize}[leftmargin=0.56in]
\item[(A.D.1)] The log-likelihood function is three times differentiable with uniformly bounded third derivative in $\theta\in \Theta$, $(x_1,x_2,\ldots,x_n) \in S$. The supremum is denoted by
\begin{equation}
\label{Sd}
S_d(n):=\sup_{\substack{\theta \in \Theta\\ \boldsymbol{x}\in S}}\left|l^{(3)}(\theta;\boldsymbol{x})\right| < \infty.
\end{equation}
\item[(A.D.2)] ${\rm E}\left[\frac{{\rm d}}{{\rm d}\theta}\log f(X_1|\theta)\right] = {\rm E}\left[\frac{{\rm d}}{{\rm d}\theta}\log f(X_i|X_{i-1},\ldots,X_{i-m};\theta)\right] = 0, \; {\rm for}\; i=2,3,\ldots,n$.
\item[(A.D.3)] With $\theta_0$, as usual, denoting the true value of the unknown parameter,
\begin{equation}
\nonumber \sqrt{n}{\rm E}\left[\hat{\theta}_n(\boldsymbol{X}) - \theta_0\right] \xrightarrow[{n\rightarrow \infty}]{{}} 0.
\end{equation}
\item[(A.D.4)] The limit of the reciprocal of $n {\rm Var}\left(\hat{\theta}_n(\boldsymbol{X})\right)$ exists and from now on, unless otherwise stated,
\begin{equation}
\nonumber 0 < i_2(\theta_0) = \lim_{n \rightarrow \infty} \frac{1}{n{\rm Var}(\hat{\theta}_n(\boldsymbol{X}))}.
\end{equation}
\end{itemize}
The following theorem gives the bound.
\vspace{0.1in}
\begin{MLEexist}
\label{Theorem_local_RD1}
Let $\left\lbrace X_i, i=1,2,\ldots,n \right\rbrace$ be an $m$-dependent sequence of identically distributed random variables with probability density (or mass) function $f(x_i|x_{i-1},\ldots,x_{i-m};\theta)$, where\linebreak $\theta \in \Theta$ and $(x_1,x_2,\ldots, x_n) \in S$, where $S$ is the support of the joint probability density (or  mass) function. Assume that $\hat{\theta}_n(\boldsymbol{X})$ exists and is unique. In addition, assume that (A.D.1)-(A.D.4) hold and that ${\rm Var}\left[l'(\theta_0;\boldsymbol{X})\right] > 0$. Let
\begin{equation}
\label{new_alpha}
\alpha:=\alpha(\theta_0,n):= \sqrt{\frac{{\rm Var}\left(l'(\theta_0;\boldsymbol{X})\right)}{{\rm Var}\left(\hat{\theta}_n(\boldsymbol{X})\right)}},
\end{equation}
which is assumed to be finite and not equal to zero. In addition, we denote by $$\xi_1 = \frac{\mathrm{d}}{\mathrm{d}\theta}\log f(X_1|\theta)\Big|_{\theta = \theta_0}\sqrt{\frac{n}{{\rm Var}(l'(\theta_0;\boldsymbol{X}))}}$$ and for $i=2,3,\ldots,n$, $$\xi_i = \frac{\mathrm{d}}{\mathrm{d}\theta}\log f(X_i|X_{i-1},\ldots,X_{i-m};\theta)\Big|_{\theta = \theta_0}\sqrt{\frac{n}{{\rm Var}(l'(\theta_0;\boldsymbol{X}))}}.$$ Then, for $Z \sim {\rm N}(0,1)$,
\begin{align}
\label{final_result_local_dependence_RD1}
\nonumber & d_{W}\left(\sqrt{n\,i_2(\theta_0)}\left(\hat{\theta}_n(\boldsymbol{X}) - \theta_0\right),Z\right) \leq \frac{2}{n^{\frac{3}{2}}}\sum_{i=1}^{n}\sum_{j \in A_i}\sum_{k \in B_i}\left[{\rm E}\left(\left(\xi_i\right)^4\right){\rm E}\left(\left(\xi_j\right)^4\right){\rm E}\left(\left(\xi_k\right)^4\right)\right]^{\frac{1}{4}}\\
\nonumber & \;\; + \frac{2}{n^{\frac{3}{2}}}\sum_{i=1}^{n}\sum_{j \in A_i}\sum_{k \in B_i}\left[{\rm E}\left(\left(\xi_i\right)^2\right){\rm E}\left(\left(\xi_j\right)^2\right){\rm E}\left(\left(\xi_k\right)^2\right)\right]^{\frac{1}{2}} + \frac{1}{n^{\frac{3}{2}}}\sum_{i=1}^{n}\left|A_i\right|\sum_{j \in A_i}\left[{\rm E}\left(\left(\xi_i\right)^2\right){\rm E}\left(\left(\xi_j\right)^4\right)\right]^{\frac{1}{2}}\\
\nonumber & \;\; + \left|\frac{\sqrt{n\,i_2(\theta_0){\rm Var}[l'(\theta_0;\boldsymbol{X})]}}{\alpha} -1 \right| + \frac{S_d(n)\sqrt{n\,i_2(\theta_0)}}{2\alpha}{\rm E}\left[\left(\hat{\theta}_n(\boldsymbol{X}) - \theta_0\right)^2\right]\\
& \;\; + \frac{\sqrt{n\,i_2(\theta_0)}}{\alpha}\sqrt{{\rm E}\left[\left(\hat{\theta}_n(\boldsymbol{X})- \theta_0\right)^2\right]}\sqrt{{\rm E}\left[\left(l''(\theta_0;\boldsymbol{X}) + \alpha\right)^2\right]}.
\end{align}
%
\end{MLEexist}
\begin{proof}
By the definition of the MLE and (A.D.1), $l'\left(\hat{\theta}_n(\boldsymbol{x});\boldsymbol{x}\right) = 0$. A second order Taylor expansion gives that
\begin{align}
\nonumber &\left(\hat{\theta}_n(\boldsymbol{X}) - \theta_0\right)l''(\theta_0;\boldsymbol{X}) = -l'(\theta_0;\boldsymbol{X}) - R_1(\theta_0;\boldsymbol{X})\\
\nonumber & \Rightarrow -\alpha\left(\hat{\theta}_n(\boldsymbol{X}) - \theta_0\right) = -l'(\theta_0;\boldsymbol{X}) - R_1(\theta_0;\boldsymbol{X}) - \left(\hat{\theta}_n(\boldsymbol{X}) - \theta_0\right)\left(l''(\theta_0;\boldsymbol{X}) + \alpha\right),
\end{align}
where
\begin{equation}
\nonumber R_1(\theta_0;\boldsymbol{X}) = \frac{1}{2}\left(\hat{\theta}_n(\boldsymbol{x})-\theta_0\right)^2l^{(3)}(\theta^*;\boldsymbol{x})
\end{equation}
is the remainder term with $\theta^{*}$ lying between $\hat{\theta}_n(\boldsymbol{x})$ and $\theta_0$. Multiplying both sides by $-\frac{\sqrt{n\,i_2(\theta_0)}}{\alpha}$,
\begin{equation}
\label{relationship_local_dep}
\sqrt{n\,i_2(\theta_0)}\left(\hat{\theta}_n(\boldsymbol{X})-\theta_0\right) = \frac{\sqrt{n\,i_2(\theta_0)}}{\alpha}\left[\vphantom{(\left(\sup_{\theta:|\theta-\theta_0|\leq\epsilon}\left|l^{(3)}(\theta;\boldsymbol{X})\right|\right)^2}l'(\theta_0;\boldsymbol{X})+R_1(\theta_0;\boldsymbol{X}) + \left(\hat{\theta}_n(\boldsymbol{X}) - \theta_0\right)\left(l''(\theta_0;\boldsymbol{X}) + \alpha\right)\vphantom{(\left(\sup_{\theta:|\theta-\theta_0|\leq\epsilon}\left|l^{(3)}(\theta;\boldsymbol{X})\right|\right)^2}\right].
\end{equation}
Applying the triangle inequality,
\begin{align}
\nonumber &\left|{\rm E}\left[h\left(\sqrt{n\,i_2(\theta_0)}\left(\hat{\theta}_n(\boldsymbol{X})-\theta_0\right)\right)\right] - {\rm E}[h(Z)]\right|\\
\label{first_term_localRD1}
&\leq \left|{\rm E}\left[h\left(\frac{\sqrt{n\,i_2(\theta_0)}l'(\theta_0;\boldsymbol{X})}{\alpha}\right)\right] - {\rm E}[h(Z)]\right|\\
\label{second_term_localRD1} &\;\;+ \left|{\rm E}\left[h\left(\sqrt{n\,i_2(\theta_0)}\left(\hat{\theta}_n(\boldsymbol{X}) - \theta_0\right)\right) - h\left(\frac{\sqrt{n\,i_2(\theta_0)}l'(\theta_0;\boldsymbol{X})}{\alpha}\right)\right]\right|.
\end{align}
\textbf{Step 1: Bound for} \eqref{first_term_localRD1}. Let, for ease of presentation $l'(\theta_0;\boldsymbol{X}) = \sum_{i=1}^{n}\tilde{\xi}_i$, where
\begin{align}
\nonumber & \tilde{\xi}_1 = \frac{\mathrm{d}}{\mathrm{d}\theta}\log f(X_1|\theta)\Big|_{\theta = \theta_0}, \qquad \tilde{\xi}_i = \frac{\mathrm{d}}{\mathrm{d}\theta}\log f(X_i|X_{i-1},\ldots,X_{i-m};\theta)\Big|_{\theta= \theta_0}\;\; {\rm for}\;\; i=2,3,\ldots,n.
\end{align}
Assumption (A.D.2) ensures that $\tilde{\xi}_i, i=1,2,\ldots,n$ have mean zero. Furthermore, for some function $g:\mathbb{R}^{m+1}\rightarrow\mathbb{R}$, it holds that $\tilde{\xi}_i = g(X_i,X_{i-1}, \ldots, X_{i-m})$ and taking into account that $\left\lbrace X_i, i=1,2,\ldots,n \right\rbrace$ is an $m$-dependent sequence, we conclude that $\left\lbrace \tilde{\xi}_i,i=1,2,\ldots,n \right\rbrace$ forms a $2m$-dependent sequence. Define now 
\begin{equation}
\label{W_local_depen}
W:=\frac{l'(\theta_0;\boldsymbol{X})}{\sqrt{{\rm Var}\left[l'(\theta_0;\boldsymbol{X})\right]}} = \sum_{i=1}^{n}\left(\frac{\xi_i}{\sqrt{n}}\right),
\end{equation}
with
$$\xi_i = \tilde{\xi}_i\sqrt{\frac{n}{{\rm Var}\left[l'(\theta_0;\boldsymbol{X})\right]}}, \;\forall i \in \left\lbrace 1,2,\ldots,n \right\rbrace.$$
It follows that $\left\lbrace \frac{\xi_i}{\sqrt{n}}, i=1,2,\ldots,n \right\rbrace$ is a random $2m$-dependent sequence with mean zero and also ${\rm Var}(W) = 1$. In addition, (LD) is satisfied with $A_j$ and $B_j$ as in \eqref{A_iB_i}. A simple triangle inequality gives that
\begin{align}
\label{first_term_local_1_RD1}
\eqref{first_term_localRD1} &\leq \left|{\rm E}[h(W)] - {\rm E}[h(Z)]\right|\\
\label{first_term_local_2_RD1}
&\;\; + \left|{\rm E}\left[h\left(\frac{\sqrt{n\,i_2(\theta_0)}l'(\theta_0;\boldsymbol{X})}{\alpha}\right)-h(W)\right]\right|.
\end{align}
Since the assumptions of Lemma \ref{Theorem_Chen_local} are satisfied for $W$ as in \eqref{W_local_depen}, one can directly use \eqref{bound_Chen_local} in order to find an upper bound for \eqref{first_term_local_1_RD1}. For \eqref{first_term_local_2_RD1}, a first order Taylor expansion of $h\left(\frac{\sqrt{n\,i_2(\theta_0)}l'(\theta_0;\boldsymbol{X})}{\alpha}\right)$ about $W$ yields
\begin{align}
\nonumber & h\left(\frac{\sqrt{n\,i_2(\theta_0)}l'(\theta_0;\boldsymbol{X})}{\alpha}\right) - h\left(\frac{l'(\theta_0;\boldsymbol{X})}{\sqrt{{\rm Var}(l'(\theta_0;\boldsymbol{X}))}}\right)\\
\nonumber & = \left(\frac{\sqrt{n\,i_2(\theta_0)}l'(\theta_0;\boldsymbol{X})}{\alpha} - \frac{l'(\theta_0;\boldsymbol{X})}{\sqrt{{\rm Var}(l'(\theta_0;\boldsymbol{X}))}}\right)h'(t_1(\boldsymbol{X})),
\end{align}
where $t_1(\boldsymbol{X})$ is between $\frac{\sqrt{n\,i_2(\theta_0)}l'(\theta_0;\boldsymbol{X})}{\alpha}$ and $\frac{l'(\theta_0;\boldsymbol{X})}{\sqrt{{\rm Var}(l'(\theta_0;\boldsymbol{X}))}}$. Therefore,
\begin{align}
\label{third_term_local_no_RD1}
\eqref{first_term_local_2_RD1} & \leq \|h'\|\left|\frac{\sqrt{n\,i_2(\theta_0)}}{\alpha}-\frac{1}{\sqrt{{\rm Var}(l'(\theta_0;\boldsymbol{X}))}}\right|{\rm E}\left|l'(\theta_0;\boldsymbol{X})\right| \leq \|h'\|\left|\frac{\sqrt{n\,i_2(\theta_0){\rm Var}(l'(\theta_0;\boldsymbol{X}))}}{\alpha} -1\right|.
\end{align}
For $h \in H_W$ as in \eqref{classWasserstein}, then $\|h'\| \leq 1$, which yields
\begin{align}
\label{chapter6mid0}
\nonumber \eqref{first_term_localRD1}&\leq \frac{2}{n^{\frac{3}{2}}}\left[\sum_{i=1}^{n}\left({\rm E}|\xi_i\eta_i\tau_i|\right) + \sum_{i=1}^{n}\left(\left|{\rm E}(\xi_i\eta_i)\right|{\rm E}|\tau_i|\right)\right] + \frac{1}{n^{\frac{3}{2}}}\sum_{i=1}^{n}{\rm E}\left|\xi_i\eta_i^2\right|\\
& \quad + \left|\frac{\sqrt{n\,i_2(\theta_0){\rm Var}(l'(\theta_0;\boldsymbol{X}))}}{\alpha} -1\right|,
\end{align}
with $\eta_i$ and $\tau_i$ as in \eqref{notation_local}. The absolute expectations in \eqref{chapter6mid0} can be difficult to bound and the first three quantities of the above bound are therefore expressed in terms of more easily calculable terms. For the first term in \eqref{chapter6mid0}, using H\"{o}lder's inequality
\begin{align}
\label{chapter6mid1}
\nonumber {\rm E}|\xi_i\eta_i\tau_i| & = {\rm E}\left|\xi_i\sum_{j\in A_i}\xi_j\sum_{k \in B_i}\xi_k\right| \leq \sum_{j \in A_i}\sum_{k \in B_i}{\rm E}\left|\xi_i\xi_j\xi_k\right| \leq \sum_{j \in A_i}\sum_{k \in B_i}\left[{\rm E}\left(\left|\xi_i\right|^3\right){\rm E}\left(\left|\xi_j\right|^3\right){\rm E}\left(\left|\xi_k\right|^3\right)\right]^{\frac{1}{3}}\\
& \leq \sum_{j \in A_i}\sum_{k \in B_i}\left[{\rm E}\left(\left(\xi_i\right)^4\right){\rm E}\left(\left(\xi_j\right)^4\right){\rm E}\left(\left(\xi_k\right)^4\right)\right]^{\frac{1}{4}}.
\end{align}
For the second term of the bound in \eqref{chapter6mid0}, the Cauchy-Schwarz inequality yields
\begin{align}
\label{chapter6mid2}
\nonumber \left|{\rm E}\left(\xi_i\eta_i\right)\right|{\rm E}\left|\tau_i\right| &= \left|{\rm E}\left(\xi_i\sum_{j \in A_i}\xi_j\right)\right|{\rm E}\left|\sum_{k \in B_i}\xi_k\right| \leq \sum_{j \in A_i}{\rm E}\left|\xi_i\xi_j\right|\sum_{k \in B_i}{\rm E}\left|\xi_k\right|\\
& \leq \sum_{j \in A_i}\sum_{k \in B_i}\left[{\rm E}\left(\left(\xi_i\right)^2\right){\rm E}\left(\left(\xi_j\right)^2\right){\rm E}\left(\left(\xi_k\right)^2\right)\right]^{\frac{1}{2}}.
\end{align}
For the third term, Jensen's inequality is employed to get that $$\left(\sum_{i \in J}|a_i|\right)^z \leq J^{z-1}\sum_{i\in J}^{}|a_i|^z,\; \forall a_i \in \mathbb{R}{\rm\;\; and\;\;}z\in\mathbb{N}$$ and therefore
\begin{align}
\label{chapter6mid3}
\nonumber {\rm E}\left|\xi_i\eta_i^2\right|&= {\rm E}\left|\xi_i\left(\sum_{j \in A_i}\xi_j\right)^2\right| \leq \left|A_i\right|{\rm E}\left|\xi_i\sum_{j \in A_i}\xi_j^2\right| \leq \left|A_i\right|\sum_{j \in A_i}{\rm E}\left|\xi_i\xi_j^2\right|\\
& \leq \left|A_i\right|\sum_{j \in A_i}\left[{\rm E}\left(\left(\xi_i\right)^2\right){\rm E}\left(\left(\xi_j\right)^4\right)\right]^{\frac{1}{2}}.
\end{align}
The results in \eqref{chapter6mid1}, \eqref{chapter6mid2} and \eqref{chapter6mid3} yield
\begin{align}
\label{bound_for_first_term_local_1}
\nonumber & \eqref{first_term_local_1_RD1} \leq \frac{2}{n^{\frac{3}{2}}}\left[\sum_{i=1}^{n}\left({\rm E}|\xi_i\eta_i\tau_i|\right) + \sum_{i=1}^{n}\left(\left|{\rm E}(\xi_i\eta_i)\right|{\rm E}|\tau_i|\right)\right] + \frac{1}{n^{\frac{3}{2}}}\sum_{i=1}^{n}{\rm E}\left|\xi_i\eta_i^2\right|\\
\nonumber & \leq \frac{2}{n^{\frac{3}{2}}}\sum_{i=1}^{n}\sum_{j \in A_i}\sum_{k \in B_i}\left[{\rm E}\left(\left(\xi_i\right)^4\right){\rm E}\left(\left(\xi_j^4\right)\right){\rm E}\left(\left(\xi_k\right)^4\right)\right]^{\frac{1}{4}}\\
& \;\; + \frac{2}{n^{\frac{3}{2}}}\sum_{i=1}^{n}\sum_{j \in A_i}\sum_{k \in B_i}\left[{\rm E}\left(\left(\xi_i\right)^2\right){\rm E}\left(\left(\xi_j\right)^2\right){\rm E}\left(\left(\xi_k\right)^2\right)\right]^{\frac{1}{2}} + \frac{1}{n^{\frac{3}{2}}}\sum_{i=1}^{n}\left|A_i\right|\sum_{j \in A_i}\left[{\rm E}\left(\left(\xi_i\right)^2\right){\rm E}\left(\left(\xi_j\right)^4\right)\right]^{\frac{1}{2}}.
\end{align}
The bound for \eqref{first_term_local_1_RD1} is now obviously a function only of ${\rm E}\left(\xi_i^2\right)$ and ${\rm E}\left(\xi_i^4\right)$.
\vspace{0.05in}
\\
\textbf{Step 2: Bound for} \eqref{second_term_localRD1}. The main tool used here is Taylor expansions. For ease of presentation, let
\begin{align}
\nonumber &\tilde{C}(\theta_0) = \tilde{C}(h,\theta_0;\boldsymbol{X}) := h\left(\sqrt{n\,i_2(\theta_0)}\left(\hat{\theta}_n(\boldsymbol{X}) - \theta_0\right)\right) - h\left(\frac{\sqrt{n\,i_2(\theta_0)}l'(\theta_0;\boldsymbol{X})}{\alpha}\right)\\
\nonumber & = h\left(\frac{\sqrt{n\,i_2(\theta_0)}\left[l'(\theta_0;\boldsymbol{X}) + R_1(\theta_0;\boldsymbol{X}) + \left(\hat{\theta}_n(\boldsymbol{X}) - \theta_0\right)\left(l''(\theta_0;\boldsymbol{X}) + \alpha\right)\right]}{\alpha}\right)\\
\nonumber &\;\;\;\;\; - h\left(\frac{\sqrt{n\,i_2(\theta_0)}l'(\theta_0;\boldsymbol{X})}{\alpha}\right)
\end{align}
using \eqref{relationship_local_dep}. A first order Taylor expansion of $h\left(\frac{\sqrt{n\,i_2(\theta_0)}\left[l'(\theta_0;\boldsymbol{X}) + R_1(\theta_0;\boldsymbol{X}) + \left(\hat{\theta}_n(\boldsymbol{X}) - \theta_0\right)\left(l''(\theta_0;\boldsymbol{X}) + \alpha\right)\right]}{\alpha}\right)$ about $\frac{\sqrt{n\,i_2(\theta_0)}l'(\theta_0;\boldsymbol{X})}{\alpha}$ yields
\begin{align}
\label{last_term_local_no_RD1}
\nonumber & \eqref{second_term_localRD1} = \left|{\rm E}\left[\tilde{C}(\theta_0)\right]\right| \leq \frac{\sqrt{n\,i_2(\theta_0)}}{\alpha}\|h'\|\left(\vphantom{(\left(\sup_{\theta:|\theta-\theta_0|\leq\epsilon}\left|l^{(3)}(\theta;\boldsymbol{X})\right|\right)^2}{\rm E}\left[\frac{1}{2}\left(\hat{\theta}_n(\boldsymbol{X})- \theta_0\right)^2\left|l^{(3)}(\theta^*;\boldsymbol{X})\right|\right]\right.\\
\nonumber & \qquad\qquad\qquad\qquad\qquad\qquad\qquad\qquad\quad\left. + {\rm E}\left|\left(\hat{\theta}_n(\boldsymbol{X}) - \theta_0\right)\left(l''(\theta_0;\boldsymbol{X}) + \alpha\right)\right|\vphantom{(\left(\sup_{\theta:|\theta-\theta_0|\leq\epsilon}\left|l^{(3)}(\theta;\boldsymbol{X})\right|\right)^2}\right)\\
& \leq \frac{\sqrt{n\,i_2(\theta_0)}}{\alpha}\|h'\|\left(\vphantom{(\left(\sup_{\theta:|\theta-\theta_0|\leq\epsilon}\left|l^{(3)}(\theta;\boldsymbol{X})\right|\right)^2}\frac{S_d(n)}{2}{\rm E}\left[\left(\hat{\theta}_n(\boldsymbol{X})- \theta_0\right)^2\right] + \sqrt{{\rm E}\left[\left(\hat{\theta}_n(\boldsymbol{X})- \theta_0\right)^2\right]}\sqrt{{\rm E}\left[\left(l''(\theta_0;\boldsymbol{X}) + \alpha\right)^2\right]}\vphantom{(\left(\sup_{\theta:|\theta-\theta_0|\leq\epsilon}\left|l^{(3)}(\theta;\boldsymbol{X})\right|\right)^2}\right),
\end{align}
where for the last step Cauchy-Schwarz inequality has been used while $S_d(n)$ is as in \eqref{Sd}. We conclude that \eqref{third_term_local_no_RD1}, \eqref{bound_for_first_term_local_1} and \eqref{last_term_local_no_RD1} yield, for $h \in H_W$, the assertion of the theorem as expressed in \eqref{final_result_local_dependence_RD1}.
\end{proof}
The following corollary specifies the result of
Theorem \ref{Theorem_local_RD1} for the simple scenario of i.i.d. random variables. This allows for a comparison with the bound given in \cite{Anastasiou_Reinert}, which is for i.i.d. random variables. The proof of the corollary is a result of simple steps and therefore only an outline is provided.
\begin{Corollary}
\label{Corollaryindependence}
Let $X_1, X_2, \ldots, X_n$ be i.i.d. random variables with probability density (or mass) function $f(x|\theta)$. Assume that $\hat{\theta}_n(\boldsymbol{X})$ exists and is unique and that (A.D.1)-(A.D.4) hold. In addition, ${\rm Var}\left[l'(\theta_0;\boldsymbol{X})\right] > 0$. For $\alpha$ as in \eqref{new_alpha} and $Z \sim {\rm N}(0,1)$,
\begin{align}
\label{final_result_independence_RD1}
\nonumber & d_{W}\left(\sqrt{n\,i_2(\theta_0)}\left(\hat{\theta}_n(\boldsymbol{X}) - \theta_0\right),Z\right) \leq \frac{5{\rm E}\left|\frac{{\mathrm{d}}}{{\mathrm{d}}\theta}\log f\left(X_1|\theta_0\right)\right|^3}{\sqrt{n}\left[{\rm Var}\left(\frac{{\mathrm{d}}}{{\mathrm{d}}\theta}\log f\left(X_1|\theta_0\right)\right)\right]^{\frac{3}{2}}}\\
\nonumber & + \left|\frac{n\sqrt{i_2(\theta_0){\rm Var}\left(\frac{{\mathrm{d}}}{{\mathrm{d}}\theta}\log f\left(X_1|\theta_0\right)\right)}}{\alpha} -1 \right| + \frac{S_d(n)\sqrt{n\,i_2(\theta_0)}}{2\alpha}{\rm E}\left[\left(\hat{\theta}_n(\boldsymbol{X}) - \theta_0\right)^2\right]\\
& \;\; + \frac{\sqrt{n\,i_2(\theta_0)}}{\alpha}\sqrt{{\rm E}\left[\left(\hat{\theta}_n(\boldsymbol{X})- \theta_0\right)^2\right]}\sqrt{{\rm E}\left[\left(l''(\theta_0;\boldsymbol{X}) + \alpha\right)^2\right]}.
\end{align}
\end{Corollary}
{\raggedright{{\textit{Outline of the proof}}.}} A similar process as the one followed in the proof of Theorem \ref{Theorem_local_RD1} shows that a bound is obtained by bounding the terms \eqref{first_term_local_1_RD1}, \eqref{first_term_local_2_RD1} and \eqref{second_term_localRD1}. For independent random variables, applying H\"{o}lder's inequality to the bound in \eqref{bound_Chen_local}, where now
$$W = \sum_{i=1}^{n}\left(\frac{\xi_i}{\sqrt{n}}\right), \qquad \xi_i = \frac{\mathrm{d}}{\mathrm{d}\theta}\log f(X_i|\theta)\Big|_{\theta = \theta_0}\sqrt{\frac{n}{{\rm Var}(l'(\theta_0;\boldsymbol{X}))}},$$ leads to $$\eqref{first_term_local_1_RD1} \leq \frac{5}{n^{\frac{3}{2}}}\sum_{i=1}^{n}{\rm E}\left|\xi_i\right|^3 = \frac{5{\rm E}\left|\frac{\mathrm{d}}{\mathrm{d}\theta}\log f(X_1|\theta_0)\right|^3}{\sqrt{n}\left[{\rm Var}\left(\frac{\mathrm{d}}{\mathrm{d}\theta}\log f(X_1|\theta_0)\right)\right]^{\frac{3}{2}}}.$$ The second term of the bound in \eqref{final_result_independence_RD1} is the special form of \eqref{third_term_local_no_RD1} for the case of i.i.d. random variables, while the last two terms are as in the result of Theorem \ref{Theorem_local_RD1}.
\begin{Remark}
The bound in \eqref{final_result_independence_RD1} is not as simple and sharp as the bound given in Theorem 2.1 of \cite{Anastasiou_Reinert}. This is expected since Corollary \ref{Corollaryindependence} is a special application of a result which was originally obtained to satisfy the assumption of local dependence for our random variables, while \cite{Anastasiou_Reinert} used directly results of Stein's method for independent random variables. In addition, the assumption (A.D.1) used for the result of Corollary \ref{Corollaryindependence} is stronger than the condition (R3) of \cite{Anastasiou_Reinert}. Using uniform boundedness of the third derivative of the log-likelihood function in (A.D.1) allows us to get bounds on the Wasserstein distance related to the MLE. On the other hand, \cite{Anastasiou_Reinert} relaxed this condition and assumed that the third derivative of the log-likelihood function is bounded in an area of $\theta_0$. This lead to bounds on the bounded Wasserstein (or Fortet-Mourier) distance; see \cite{Nourdin} for a definition of this metric.
\end{Remark} 
\section{Example: 1-dependent normal random variables}
\label{sec:example_local_normal}
To illustrate the general results, as an example assume that we have a sequence $\left\lbrace S_1,S_2,\ldots,S_n\right\rbrace$ of random variables where for $k\in\mathbb{Z}^+$ and $\forall j \in \left\lbrace 1,2,\ldots,n \right\rbrace$,
\begin{equation}
\nonumber S_j = \sum_{i=(j-1)k}^{jk}X_i,
\end{equation}
for $X_i, i=0,1,2,\ldots,nk\;$ i.i.d. random variables from the ${\rm N}(\mu,\sigma^2)$ distribution with $\mu=\theta\in\mathbb{R}$ being the unknown parameter and $\sigma^2$ is known. Hence $S_j$ and $S_{j+1}$ share one summand, $X_{jk}$. For $\delta \in \mathbb{Z} \setminus \left\lbrace 0 \right\rbrace$, we have that
\begin{equation}
\nonumber {\rm Cov}(S_i, S_{i+\delta}) =
\begin{cases}
{\rm Var}(X_1) = \sigma^2, & \text{if $|\delta| = 1$}\\
0, & \text{if $|\delta| > 1$}.
\end{cases}
\end{equation}
Therefore, $\left\lbrace S_i \right\rbrace_{i=1,2,\ldots,n}$ is a 1-dependent sequence of random variables. Furthermore,
\begin{equation}
\label{distributionS1}
S_i \sim {\rm N}((k+1)\theta,(k+1)\sigma^2)
\end{equation}
as it is a sum of $k+1$ independent normally distributed random variables with mean $\theta$ and variance $\sigma^2$. As
\begin{equation}
\nonumber \rho = \frac{{\rm Cov}(S_{i-1},S_i)}{\sqrt{{\rm Var}(S_{i-1}){\rm Var}(S_i)}} = \frac{\sigma^2}{(k+1)\sigma^2} = \frac{1}{k+1}, \; \forall i\in \left\lbrace 2,3,\ldots,n \right\rbrace,
\end{equation}
it is standard, see \cite[p.177]{Casella}, that for $i = 2,3,\ldots,n$
\begin{equation}
\label{distributionSi}
\left(S_i|S_{i-1}=s_{i-1}\right) \sim {\rm N}\left((k+1)\theta + \frac{1}{k+1}\left(s_{i-1}-(k+1)\theta\right),\frac{k(k+2)}{k+1}\sigma^2\right).
\end{equation}
After basic steps, the likelihood function for the parameter $\theta$ under $\boldsymbol{S} = (S_1,S_2, \ldots, S_n)$ is
\begin{align}
\nonumber L(\theta;\boldsymbol{S}) & = f(S_1|\theta)\prod_{i=2}^{n}f(S_i|S_{i-1};\theta)\\
\nonumber & = \frac{(k+1)^{\frac{n-1}{2}}}{\sqrt{2\pi(k+1)\sigma^2}(2\pi k(k+2)\sigma^2)^{\frac{n-1}{2}}}\exp\left\lbrace\vphantom{(\left(\sup_{\theta:|\theta-\theta_0|\leq\epsilon}\left|l^{(3)}(\theta;\boldsymbol{X})\right|\right)^2} -\frac{\left(S_1-(k+1)\theta\right)^2}{2(k+1)\sigma^2}\right.\\
\nonumber & \left.\qquad\qquad\quad -\frac{k+1}{2k(k+2)\sigma^2}\sum_{i=2}^{n}\left(S_i-\left(\left(k+1\right)\theta+\frac{1}{k+1}(S_{i-1}-(k+1)\theta)\right)\right)^2\vphantom{(\left(\sup_{\theta:|\theta-\theta_0|\leq\epsilon}\left|l^{(3)}(\theta;\boldsymbol{X})\right|\right)^2}\right\rbrace.
\end{align} 
Having this closed-form expression for the likelihood allows us to derive the MLE under this local dependence structure. The unique MLE for $\theta$ is
\begin{equation}
\label{MLE_local}
\hat{\theta}_n(\boldsymbol{S}) = \frac{k\sum_{i=1}^{n}S_i + S_1 + S_n}{(nk+2)(k+1)}.
\end{equation}
In addition, the first three derivatives of the log-likelihood function are
\begin{align}
\label{score_local_example}
\nonumber & l'(\theta;\boldsymbol{S}) = \frac{1}{(k+2)\sigma^2}\left\lbrace k\sum_{i=1}^{n}S_i + S_1 + S_n -(k+1)(nk+2)\theta\right\rbrace\\
\nonumber & l''(\theta;\boldsymbol{S}) = -\frac{(nk+2)(k+1)}{(k+2)\sigma^2}\\
& l^{(3)}(\theta;\boldsymbol{S}) = 0.
\end{align}
The following Corollary gives the upper bound on the Wasserstein distance between the distribution of $\hat{\theta}_n(\boldsymbol{S})$ and the normal distribution.
\vspace{0.1in}
\begin{Corollary}
Let $S_1, S_2, \ldots S_n$ be a 1-dependent sequence of random variables with \linebreak $S_i \sim {\rm N}((k+1)\theta,(k+1)\sigma^2)$. The conditions (A.D.1)-(A.D.4) hold. For $Z \sim {\rm N}(0,1)$ and $i_2(\theta_0) = \frac{(k+1)^2}{(k+3)\sigma^2}$,
\begin{align}
\nonumber & d_W\left(\sqrt{n\,i_2(\theta_0)}\left(\hat{\theta}_n(\boldsymbol{S})-\theta_0\right),Z\right)\leq 339(n-5)\left[\frac{k(k+1)(k+2)}{(nk^3+(3n+2)k^2+10k+2)}\right]^{\frac{3}{2}}\\
\nonumber & \;\; + \frac{(k+1)^{\frac{3}{2}}(k+2)^{\frac{3}{2}}}{(nk^3+(3n+2)k^2+10k+2)^{\frac{3}{2}}}\left\lbrace\vphantom{(\left(\sup_{\theta:|\theta-\theta_0|\leq\epsilon}\left|l^{(3)}(\theta;\boldsymbol{X})\right|\right)^2} \left(1+3^{\frac{3}{4}}\right)\left(2\sqrt{k+2}(37k+2)+4\sqrt{k}(61k+8)\right)\right.\\
\nonumber & \qquad\quad\qquad\qquad\qquad\qquad\qquad\qquad\qquad\left. +\sqrt{3}\left(3\sqrt{k+2}(k+1)+\sqrt{k}(91k+18)\right)\vphantom{(\left(\sup_{\theta:|\theta-\theta_0|\leq\epsilon}\left|l^{(3)}(\theta;\boldsymbol{X})\right|\right)^2}\right\rbrace\\
\nonumber & \;\; + \left|\left(1-\frac{2}{nk+2}\right)\left[\frac{k+3+\frac{2}{n}+\frac{10}{nk}+\frac{2}{nk^2}}{k+3}\right]^{\frac{1}{2}} - 1\right|.
\end{align}
\end{Corollary}
\begin{Remark}
The order of the bound with respect to the sample size is $\frac{1}{\sqrt{n}}$.
\end{Remark}
\begin{proof}
We first check that the assumptions (A.D.1)-(A.D.4) are satisfied. The first assumption is satisfied from \eqref{score_local_example} with $S_d(n) = 0$. From \eqref{distributionS1} and \eqref{distributionSi}, simple steps yield ${\rm E}\left[\frac{{\rm d}}{{\rm d}\theta}\log f(S_1|\theta)\right] = {\rm E}\left[\frac{{\rm d}}{{\rm d}\theta}\log f(S_i|S_{i-1};\theta)\right] = 0$ and thus (A.D.2) holds. The assumption (A.D.3) is also satisfied since, using \eqref{distributionS1} and \eqref{MLE_local},
\begin{equation}
\nonumber {\rm E}\left[\hat{\theta}_n(\boldsymbol{S})\right] = \frac{nk(k+1)\theta_0 + 2(k+1)\theta_0}{(nk+2)(k+1)} = \theta_0.
\end{equation}
To show that (A.D.4) holds, we first calculate
\begin{align}
\label{midstepMLEexamplelocal}
\nonumber {\rm Var}\left[\hat{\theta}_n(\boldsymbol{S})\right]& = \frac{1}{(nk+2)^2(k+1)^2}{\rm Var}\left(k\sum_{i=1}^{n}S_i + S_1 + S_n\right)\\
\nonumber & = \frac{1}{(nk+2)^2(k+1)^2}\left\lbrace\vphantom{(\left(\sup_{\theta:|\theta-\theta_0|\leq\epsilon}\left|l^{(3)}(\theta;\boldsymbol{X})\right|\right)^2} k^2{\rm Var}\left(\sum_{i=1}^{n}S_i\right) +{\rm Var}(S_1) + {\rm Var}(S_n) + 2k{\rm Cov}\left(S_1,\sum_{i=1}^{n}S_i\right)\right.\\
& \qquad\qquad\qquad\qquad\qquad\quad\left. + 2k{\rm Cov}\left(S_n,\sum_{i=1}^{n}S_i\right)\vphantom{(\left(\sup_{\theta:|\theta-\theta_0|\leq\epsilon}\left|l^{(3)}(\theta;\boldsymbol{X})\right|\right)^2}\right\rbrace.
\end{align}
From \eqref{distributionS1}, ${\rm Var}(S_i) = (k+1)\sigma^2, \forall i \in \left\lbrace 1,2,\ldots,n \right\rbrace$. In addition, since $\left\lbrace S_i\right\rbrace_{i=1,2,\ldots,n}$ is a 1-dependent sequence of random variables,
\begin{align}
\label{var_score_mid_2}
\nonumber & {\rm Var}\left(\sum_{i=1}^{n}S_i\right) = n{\rm Var}(S_1)+2(n-1){\rm Cov}(S_1,S_2) = n(k+1)\sigma^2 + 2(n-1)\sigma^2\\
& {\rm Cov}\left(S_1,\sum_{i=1}^{n}S_i\right) = {\rm Var}(S_1) + {\rm Cov}(S_1,S_2) = (k+2)\sigma^2.
\end{align}
Applying the above results of \eqref{var_score_mid_2} to \eqref{midstepMLEexamplelocal} gives that
\begin{equation}
\label{varianceMLElocalfin}
{\rm Var}\left[\hat{\theta}_n(\boldsymbol{S})\right] = \frac{\sigma^2\left(nk^3 +3nk^2+2k^2+10k+2\right)}{(nk+2)^2(k+1)^2}.
\end{equation}
Therefore,
\begin{align}
\label{Fisher_example_noRD1}
\nonumber i_2(\theta_0) &= \lim_{n \rightarrow \infty} \frac{1}{n{\rm Var}\left(\hat{\theta}_n(\boldsymbol{S})\right)} = \lim_{n \rightarrow \infty}\frac{(nk+2)^2(k+1)^2}{n\sigma^2(nk^3+3nk^2+2k^2+10k+2)}\\
& = \lim_{n \rightarrow \infty} \frac{n^2(k+1)^2\left(k^2+ \frac{4k}{n}+\frac{4}{n^2}\right)}{n^2\sigma^2\left(k^3+3k^2 + \frac{2k^2}{n} + \frac{10k}{n} + \frac{2}{n}\right)} = \frac{(k+1)^2}{(k+3)\sigma^2} > 0,
\end{align}
which shows that (A.D.4) is satisfied. To obtain $\alpha$ as defined in \eqref{new_alpha}, the variance of the score function is calculated which, after simple steps and using \eqref{var_score_mid_2}, is
\begin{align}
\label{var_score_fin}
{\rm Var}\left[l'(\theta_0;\boldsymbol{S})\right] = \frac{1}{(k+2)^2\sigma^2}\left[nk^3+3nk^2+2k^2+10k+2\right].
\end{align}
The above result and \eqref{varianceMLElocalfin} yield
\begin{equation}
\label{alpha_example_noRD1}
\alpha = \sqrt{\frac{{\rm Var}\left[l'(\theta_0;\boldsymbol{S})\right]}{{\rm Var}\left[\hat{\theta}_n(\boldsymbol{S})\right]}} = \sqrt{\frac{(nk+2)^2(k+1)^2}{\sigma^4(k+2)^2}} = \frac{(nk+2)(k+1)}{(k+2)\sigma^2}.
\end{equation}
For $\xi_1= \frac{\mathrm{d}}{\mathrm{d}\theta}\log f(S_1|\theta)\Big|_{\theta=\theta_0}\sqrt{\frac{n}{{\rm Var}\left[l'(\theta_0;\boldsymbol{S})\right]}}$, $\xi_i = \frac{\mathrm{d}}{\mathrm{d}\theta}\log f(S_i|S_{i-1};\theta)\Big|_{\theta=\theta_0}\sqrt{\frac{n}{{\rm Var}\left[l'(\theta_0;\boldsymbol{S})\right]}}, i=2,3,\ldots,n$, using \eqref{var_score_fin} and \eqref{distributionS1}, we get that $$\xi_1 = \frac{\sqrt{n}(k+2)[S_1 - (k+1)\theta_0]}{\sigma\sqrt{nk^3+(3n+2)k^2+10k+2}}$$ and therefore
\begin{align}
\label{results_x1}
\nonumber & {\rm E}\left(\xi_1^2\right) = \frac{n(k+2)^2{\rm E}(S_1-(k+1)\theta)^2}{(nk^3+(3n+2)k^2+10k+2)\sigma^2}= \frac{n(k+2)^2(k+1)}{nk^3+(3n+2)k^2+10k+2}\\
& {\rm E}\left(\xi_1^4\right) = \frac{n^2(k+2)^4{\rm E}(S_1-(k+1)\theta)^4}{(nk^3+(3n+2)k^2+10k+2)^2\sigma^4} = \frac{3n^2(k+2)^4(k+1)^2}{(nk^3+(3n+2)k^2+10k+2)^2}.
\end{align}
Furthermore, for $i=2,3,\ldots,n$, the results in \eqref{var_score_fin} and \eqref{distributionSi} yield
\begin{equation}
\nonumber \xi_i = \frac{\sqrt{n}(k+1)\left[S_i - \left((k+1)\theta+\frac{1}{k+1}\left(S_{i-1}-(k+1)\theta\right)\right)\right]}{\sigma\sqrt{nk^3+(3n+2)k^2+10k+2}}
\end{equation}
so that
\begin{align}
\label{results_xi}
\nonumber & {\rm E}\left(\xi_i^2\right) = \frac{n(k+1)^2{\rm E}\left[S_i - \left((k+1)\theta+\frac{1}{k+1}\left(S_{i-1}-(k+1)\theta\right)\right)\right]^2}{\sigma^2(nk^3+(3n+2)k^2+10k+2)}  = \frac{nk(k+1)(k+2)}{nk^3+(3n+2)k^2+10k+2}\\
\nonumber & {\rm E}\left(\xi_i^4\right) = \frac{n^2(k+1)^4{\rm E}\left[S_i - \left((k+1)\theta+\frac{1}{k+1}\left(S_{i-1}-(k+1)\theta\right)\right)\right]^4}{\sigma^4(nk^3+(3n+2)k^2+10k+2)^2}\\
& \qquad\;\;\; = \frac{3n^2k^2(k+1)^2(k+2)^2}{(nk^3+(3n+2)k^2+10k+2)^2}.
\end{align}
The first three terms of the general bound \eqref{final_result_local_dependence_RD1} are now calculated. These are denoted from now on by
\begin{align}
\nonumber &Q_v = Q_v(k,n):=\frac{1}{n^{\frac{3}{2}}}\left\lbrace\vphantom{(\left(\sup_{\theta:|\theta-\theta_0|\leq\epsilon}\left|l^{(3)}(\theta;\boldsymbol{X})\right|\right)^2}2\sum_{j \in A_v}\sum_{l \in B_v}\left[{\rm E}\left(\left(\xi_v\right)^2\right){\rm E}\left(\left(\xi_j\right)^2\right){\rm E}\left(\left(\xi_l\right)^2\right)\right]^{\frac{1}{2}}\right.\\
\nonumber &\left.\; + 2\sum_{j \in A_v}\sum_{l \in B_v}\left[{\rm E}\left(\left(\xi_v\right)^4\right){\rm E}\left(\left(\xi_j\right)^4\right){\rm E}\left(\left(\xi_l\right)^4\right)\right]^{\frac{1}{4}} + \left|A_v\right|\sum_{j \in A_v}\left[{\rm E}\left(\left(\xi_v\right)^2\right){\rm E}\left(\left(\xi_j\right)^4\right)\right]^{\frac{1}{2}}\vphantom{(\left(\sup_{\theta:|\theta-\theta_0|\leq\epsilon}\left|l^{(3)}(\theta;\boldsymbol{X})\right|\right)^2}\right\rbrace.
\end{align}
Our approach is split depending on whether 1 is an element of either $A_i$ or $B_i$ as defined in \eqref{A_iB_i} for $i \in \left\lbrace 1,2,\ldots,n \right\rbrace$.\\
{\textbf{Case 1:}} $i=6,7,\ldots,n$. Using the results in \eqref{results_xi} and since $\left|A_i\right| \leq 5$, $\left|B_i\right|\leq 9$, $\forall i \in \left\lbrace 6,7,\ldots,n \right\rbrace$,
\begin{align}
\label{case1_i_i}
\nonumber & Q_i \leq \frac{1}{n^{\frac{3}{2}}}\left\lbrace\vphantom{(\left(\sup_{\theta:|\theta-\theta_0|\leq\epsilon}\left|l^{(3)}(\theta;\boldsymbol{X})\right|\right)^2}90\left[{\rm E}\left(\left(\xi_2\right)^4\right)\right]^{\frac{3}{4}} + 90\left[{\rm E}\left(\left(\xi_2\right)^2\right)\right]^{\frac{3}{2}} + 25\left[{\rm E}\left(\left(\xi_2\right)^2\right){\rm E}\left(\left(\xi_2\right)^4\right)\right]^{\frac{1}{2}}\vphantom{(\left(\sup_{\theta:|\theta-\theta_0|\leq\epsilon}\left|l^{(3)}(\theta;\boldsymbol{X})\right|\right)^2}\right\rbrace\\
\nonumber & = \left[\frac{k(k+1)(k+2)}{(nk^3+(3n+2)k^2+10k+2)}\right]^{\frac{3}{2}}\left(90(3)^{\frac{3}{4}}+90+25\sqrt{3}\right)\\
& < 339\left[\frac{k(k+1)(k+2)}{(nk^3+(3n+2)k^2+10k+2)}\right]^{\frac{3}{2}}.
\end{align}
Issues arise due to $\xi_1$ not having the same distribution as $\xi_i$ for $i \in \left\lbrace 2,3,\ldots,n \right\rbrace$. There are hence five more special cases corresponding to $i=1,2,\ldots,5$. These cases are treated separately.\\ 
{\textbf{Case 2:}} $i=1$. For $A_1 = \left\lbrace 1,2,3\right\rbrace$ and $B_1 = \left\lbrace 1,2,\ldots,5\right\rbrace$, the results in \eqref{results_x1} and \eqref{results_xi} yield
\begin{align}
\label{case1_i_1}
\nonumber & Q_1 = \frac{1}{n^{\frac{3}{2}}}\left\lbrace\vphantom{(\left(\sup_{\theta:|\theta-\theta_0|\leq\epsilon}\left|l^{(3)}(\theta;\boldsymbol{X})\right|\right)^2}2\left[\left[{\rm E}\left(\left(\xi_1\right)^2\right)\right]^{\frac{3}{2}} + 6{\rm E}\left(\left(\xi_1\right)^2\right)\left[{\rm E}\left(\left(\xi_2\right)^2\right)\right]^{\frac{1}{2}} + 8{\rm E}\left(\left(\xi_2\right)^2\right)\left[{\rm E}\left(\left(\xi_1\right)^2\right)\right]^{\frac{1}{2}}\right]\right.\\
\nonumber &\left.\; + 2\left[\left[{\rm E}\left(\left(\xi_1\right)^4\right)\right]^{\frac{3}{4}} + 6\left[{\rm E}\left(\left(\xi_1\right)^4\right)\right]^{\frac{1}{2}}\left[{\rm E}\left(\left(\xi_2\right)^4\right)\right]^{\frac{1}{4}} + 8\left[{\rm E}\left(\left(\xi_1\right)^4\right)\right]^{\frac{1}{4}}\left[{\rm E}\left(\left(\xi_2\right)^4\right)\right]^{\frac{1}{2}}\right]\right.\\
\nonumber & \left.\; + 3\left[\left[{\rm E}\left(\left(\xi_1\right)^2\right)\right]^{\frac{1}{2}}\left(\left[{\rm E}\left(\left(\xi_1\right)^4\right)\right]^{\frac{1}{2}} + 2 \left[{\rm E}\left(\left(\xi_2\right)^4\right)\right]^{\frac{1}{2}}\right)\right]\vphantom{(\left(\sup_{\theta:|\theta-\theta_0|\leq\epsilon}\left|l^{(3)}(\theta;\boldsymbol{X})\right|\right)^2}\right\rbrace\\
& = \frac{2(k+1)^{\frac{3}{2}}(k+2)^2}{(nk^3+(3n+2)k^2+10k+2)^{\frac{3}{2}}}\left\lbrace\vphantom{(\left(\sup_{\theta:|\theta-\theta_0|\leq\epsilon}\left|l^{(3)}(\theta;\boldsymbol{X})\right|\right)^2} \left(9k+2+6\sqrt{k(k+2)}\right)\left(1+3^{\frac{3}{4}}\right) +3\sqrt{3}(k+1)\vphantom{(\left(\sup_{\theta:|\theta-\theta_0|\leq\epsilon}\left|l^{(3)}(\theta;\boldsymbol{X})\right|\right)^2}\right\rbrace.
\end{align}
{\textbf{Case 3:}} $i=2$. For $A_2 = \left\lbrace 1,2,3,4\right\rbrace$ and $B_2 = \left\lbrace 1,2,\ldots,6\right\rbrace$, a similar approach as the one in Case 2 yields
\begin{equation}
\label{case1_i_2}
Q_2 = \frac{4\sqrt{k}(k+1)^{\frac{3}{2}}(k+2)^{\frac{3}{2}}}{(nk^3+(3n+2)k^2+10k+2)^{\frac{3}{2}}}\left\lbrace\vphantom{(\left(\sup_{\theta:|\theta-\theta_0|\leq\epsilon}\left|l^{(3)}(\theta;\boldsymbol{X})\right|\right)^2} \left(8k+1+4\sqrt{k(k+2)}\right)\left(1+3^{\frac{3}{4}}\right) +2\sqrt{3}(2k+1)\vphantom{(\left(\sup_{\theta:|\theta-\theta_0|\leq\epsilon}\left|l^{(3)}(\theta;\boldsymbol{X})\right|\right)^2}\right\rbrace.
\end{equation}
{\textbf{Case 4:}} $i=3$. Following the same steps as in Case 3, now for $A_3 = \left\lbrace 1,2,\ldots,5\right\rbrace$ and $B_3 = \left\lbrace 1,2,\ldots,7\right\rbrace$, the results in \eqref{results_x1} and \eqref{results_xi} give that
\begin{equation}
\label{case1_i_3}
Q_3 = \frac{\sqrt{k}(k+1)^{\frac{3}{2}}(k+2)^{\frac{3}{2}}}{(nk^3+(3n+2)k^2+10k+2)^{\frac{3}{2}}}\left\lbrace\vphantom{(\left(\sup_{\theta:|\theta-\theta_0|\leq\epsilon}\left|l^{(3)}(\theta;\boldsymbol{X})\right|\right)^2} 2\left(25k+2+10\sqrt{k(k+2)}\right)\left(1+3^{\frac{3}{4}}\right)+5\sqrt{3}(5k+2)\vphantom{(\left(\sup_{\theta:|\theta-\theta_0|\leq\epsilon}\left|l^{(3)}(\theta;\boldsymbol{X})\right|\right)^2}\right\rbrace.
\end{equation}
{\textbf{Case 5:}} $i=4$. In this case, $A_4 = \left\lbrace 2,3,\ldots,6\right\rbrace, B_4 = \left\lbrace 1,2,\ldots,8\right\rbrace$, which lead to
\begin{equation}
\label{case1_i_4}
Q_4 = \frac{5k(k+1)^{\frac{3}{2}}(k+2)^{\frac{3}{2}}}{(nk^3+(3n+2)k^2+10k+2)^{\frac{3}{2}}}\left\lbrace 2\left[\sqrt{k+2}+7\sqrt{k}\right]\left(1+ 3^{\frac{3}{4}}\right) + 5\sqrt{3k}\right\rbrace.
\end{equation}
{\textbf{Case 6:}} $i=5$. Now $A_5 = \left\lbrace 3,4,\ldots,7\right\rbrace$ and $B_5 = \left\lbrace 1,2,\ldots,9\right\rbrace$ to obtain that
\begin{equation}
\label{case1_i_5}
Q_5 = \frac{5k(k+1)^{\frac{3}{2}}(k+2)^{\frac{3}{2}}}{(nk^3+(3n+2)k^2+10k+2)^{\frac{3}{2}}}\left\lbrace 2\left[\sqrt{k+2}+8\sqrt{k}\right]\left(1+ 3^{\frac{3}{4}}\right) + 5\sqrt{3k}\right\rbrace.
\end{equation}
The sum of the results of \eqref{case1_i_1}, \eqref{case1_i_2}, \eqref{case1_i_3}, \eqref{case1_i_4} and \eqref{case1_i_5} with $(n-5)$ times the bound in \eqref{case1_i_i} consists an upper bound for the first three terms of the general upper bound as expressed in \eqref{final_result_local_dependence_RD1}. For the fourth term of the general upper bound, \eqref{Fisher_example_noRD1}, \eqref{var_score_fin} and \eqref{alpha_example_noRD1} yield
\begin{align}
\label{third_term_local_example}
\nonumber \left|\frac{\sqrt{n\,i_2(\theta_0){\rm Var}[l'(\theta_0;\boldsymbol{X})]}}{\alpha} -1 \right| & = \left|\frac{nk}{nk+2}\left[\frac{k+3+\frac{2}{n}+\frac{10}{nk}+\frac{2}{nk^2}}{k+3}\right]^{\frac{1}{2}} - 1\right|\\
& = \left|\left(1-\frac{2}{nk+2}\right)\left[\frac{k+3+\frac{2}{n}+\frac{10}{nk}+\frac{2}{nk^2}}{k+3}\right]^{\frac{1}{2}} - 1\right|.
\end{align}
The fifth term of the bound in \eqref{final_result_local_dependence_RD1} involves the calculation of $S_d(n)$, which is equal to zero from \eqref{score_local_example}. Therefore, the fifth term of the general upper bound vanishes for this example. For the last term we have from \eqref{score_local_example} that ${\rm E}\left[l''(\theta_0;\boldsymbol{S})\right] = -\frac{(nk+2)(k+1)}{(k+2)\sigma^2} = -\alpha$ and therefore
\begin{align}
\nonumber \sqrt{{\rm E}\left[\left(\hat{\theta}_n(\boldsymbol{S}) - \theta_0\right)^2\right]{\rm E}\left[\left(l''(\theta_0;\boldsymbol{S})+\alpha\right)^2\right]} = \sqrt{{\rm E}\left[\left(\hat{\theta}_n(\boldsymbol{S}) - \theta_0\right)^2\right]{\rm Var}\left[l''(\theta_0;\boldsymbol{S})\right]} = 0.
\end{align}
The results of Case 1 - Case 6 and \eqref{third_term_local_example} give the assertion of the corollary.
\end{proof}
{\raggedright{\textbf{\textit{Remarks}}}} Several exciting paths lead from the work explained in this paper. Firstly, treating the case of a vector parameter is the next reasonable step to go (work in progress). Furthermore, other types of dependence structure between the random variables (or vectors) could be investigated to get bounds for the distributional distance of interest. In addition, our theoretical results can be very useful when it comes to applications that satisfy the assumed dependence structure for the data.

\

\noindent ACKNOWLEDGMENTS\vspace{2mm}

\noindent This research occurred whilst Andreas Anastasiou was studying for a D.Phil. at the University of Oxford, supported by a Teaching Assistantship Bursary from the Department of Statistics, University of Oxford, and the Engineering and Physical Sciences Research Council (EPSRC) grant EP/K503113/1. Andreas Anastasiou is currently funded by EPSRC Fellowship EP/L014246/1. The author would like to thank Gesine Reinert for insightful comments and suggestions. 

\vspace{0.2in}
\raggedright{{\textbf{References}}}
\bibliographystyle{elsarticle-num}
\bibliography{Paper_local_dependence_after_Gesine_first_version}

\end{document}